\documentclass[11pt]{amsart} 
\usepackage{amsmath, amssymb,
  amsthm,fullpage, epic, eepic, graphics, mathrsfs, hyperref}
\usepackage[all]{xypic} \newtheorem{thm}[equation]{Theorem}
\newtheorem{prop}[equation]{Proposition}

\newtheorem{cor}[equation]{Corollary}
\newtheorem{lemma}[equation]{Lemma}

\newtheorem{claim}[equation]{Claim}
\newtheorem{ques}[equation]{Question}

\theoremstyle{definition}
\newtheorem{defn}[equation]{Definition}

\newtheorem{rem}[equation]{Remark}

\theoremstyle{remark}

\newtheorem{ntn}[equation]{Notation}

\makeatletter
\renewcommand{\subsection}{\@startsection{subsection}{2}{0pt}{-3ex
plus -1ex minus -0.2ex}{-2mm plus -0pt minus
-2pt}{\normalfont\bfseries}}
\renewcommand{\subsubsection}{\@startsection{subsubsection}{2}{0pt}{-3ex
plus -1ex minus -0.2ex}{-2mm plus -0pt minus
-2pt}{\normalfont\bfseries}} \makeatother

\numberwithin{equation}{section}

\newdir^{ (}{{}*!/-5pt/@^{(}}

\newcommand{\erem}{\hfill$\lozenge$\end{rem}\vskip 3pt }

\newcommand{\beq}{\begin{equation}\label}
\newcommand{\eeq}{\end{equation}}

\newcommand{\into}{\hookrightarrow}

\newcommand{\Spec}{\operatorname{Spec}}

 \newcommand{\en}{\enspace }

\def\R{\mathbb{R}}
\def\C{\mathbb{C}}
\def\Z{\mathbb{Z}}

\newcommand{\Sym}{\operatorname{Sym}}

\def\Z{{\mathbb Z}}

\def\1{\mathbf{1}}

\begin{document}
\title{\en Normality and quadraticity for special ample line bundles
  on toric varieties arising from root systems} \author{Q\"endrim
  R. Gashi and Travis Schedler}

\begin{abstract}
  We prove that special ample line bundles on toric varieties arising
  from root systems are projectively normal. Here the maximal cones of
  the fans correspond to the Weyl chambers, and special means that the
  bundle is torus-equivariant such that the character of the line
  bundle that corresponds to a maximal Weyl chamber is dominant with
  respect to that chamber.  Moreover, we prove that the associated
  semigroup rings are quadratic.
\end{abstract}

\maketitle
{
\setcounter{tocdepth}{1} \tableofcontents}

\section{Introduction and statement of main results}

\subsection{Notation} Let $\Phi$ be an irreducible, reduced root
system of rank $n$. We write $Y$, $X$, $X^{\vee}$, and $Y^{\vee}$ for
the root, weight, coroot, and coweight lattice, respectively. Note
that $Y \subseteq X$ and $X^{\vee} \subseteq Y^{\vee}$. We write
$\Delta$ for the set of roots of $\Phi$.

Denote by $W$ the Weyl group of $\Phi$. Let $F$ be the Weyl fan in
$Y^{\vee} \otimes_{\Z} \R$ and $F_n \subset F$ the set of chambers
(i.e., cones of maximal dimension) in $F$. The elements of $F_n$ are
the Weyl chambers cut out by the root hyperplanes of $\Phi$.

We study the complex toric variety $V$, whose fan is $F$ and 
initial lattice is $Y^{\vee}$. The toric variety $V$ has been studied
by many authors, e.g., \cite{procesi}, \cite{klyachko},
\cite{dabrowski-normality}, \cite{carrell-kurth},
\cite{carrell-kuttler}, \cite{qendrim2}, and \cite{qendrim}. It is a
smooth, projective toric variety for the torus $T_1= \Spec(\C[Y])
\simeq(\C^{\times})^n$. 
It is a well known fact (although we do not use it) that those toric
varieties are closures of generic torus orbits in the flag variety
$G/B$, where $G$ is the reductive group associated to $\Phi$ and $B
\subseteq G$ is a Borel subgroup.

Since $T_1$ acts on $V$, the torus $T= \Spec(\C[X])$ also acts on $V$
via the canonical projection $T \twoheadrightarrow T_1$. Let
$\mathcal{L}$ be a $T$-equivariant ample line bundle on $V$. Such line
bundles (or, more precisely, the isomorphism classes thereof) are in
one-to-one correspondence (see, e.g., \cite{fulton}) with convex
polytopes $P \subset X {\otimes}_{\Z} \R$ satisfying the following
property: The vertices of $P$ are given by a set $\{ \mu_{\sigma} :
\sigma \in F_n\} \subset X$, and for any two vertices $\mu_{\sigma}$
and $\mu_{\sigma '}$ of $P$, where $\sigma$ and $\sigma'$ are adjacent
chambers, $\mu_{\sigma} - \mu_{\sigma'} = r_{\sigma, \sigma'}
\, \alpha_{\sigma,\sigma'}$, for some number $r_{\sigma, \sigma'} \in
\Z_{> 0}$, where $\alpha_{\sigma,\sigma'} \in \Delta$ is the unique
root that is positive on $\sigma$ and negative on $\sigma'$. Such
polytopes are called ``ample.'' (In, e.g., \cite{arthur}, the 
sets $\{\mu_\sigma\}$ are called
``strictly positive orthogonal sets'' in this case.)

We denote by $\Lambda(P)$ the set of points $x \in P \cap X$ whose
image in $X/Y$ coincides with the images in $X/Y$ of the vertices of
$P$, i.e., $\Lambda(P) = P \cap \{y+\mu_\sigma \mid y \in Y\}$ for any
choice of Weyl chamber $\sigma$. Note that the character $x \in X$
occurs in $H^0(X, \mathcal{L})$ if and only if $x \in \Lambda(P)$,
where $P$ is the polytope corresponding to $\mathcal{L}$ (and then it
occurs with multiplicity one): see, e.g., \cite[\S 23.1,
p.~496]{akot}.

To every chamber $\sigma \in F_n$ there corresponds a basis
$\{\alpha_{i, \sigma}: i \in I\} \subseteq \Delta$ of $Y$ consisting
of elements of $\Delta$, where $I := \{ 1, \ldots, n\}$ (in other
words, a choice of simple roots). We say that an element
$x \in X \otimes_{\Z} \R$ is $\sigma$-dominant if $\langle x,
\alpha_{i, \sigma}^{\vee} \rangle \geq 0$, $\forall i \in I$. Here
$\langle \, , \rangle $ is the usual bilinear pairing $X \times
X^{\vee} \rightarrow \Z$, extended to $(X \otimes_{\Z} \R) \times
X^{\vee} \rightarrow \R$, and $\alpha_{i, \sigma}^{\vee}$ is the
coroot in $\Phi$ corresponding to $\alpha_{i, \sigma}$.

We impose a restriction on the type of polytopes $P$ that we consider:

\begin{itemize} \item[$(\dagger)$] For every $\sigma \in F_n$, the
  element $\mu_{\sigma}$ is $\sigma$-dominant.
\end{itemize}

Following Kottwitz (\emph{op.~cit.}, \S 12.9, p.~44), we call ample
polytopes satisfying the property $(\dagger)$
\emph{special}.\footnote{More generally, Kottwitz defines
  \emph{special orthogonal sets}, where an orthogonal set is a
  collection $\{\mu_\sigma\}$ where $\mu_\sigma - \mu_\sigma' =
  r_{\sigma, \sigma'} \, \alpha_{\sigma,\sigma'}$ for $r_{\sigma,
    \sigma'} \in \Z$, not necessarily positive. A special orthogonal
  set is then one satisfying $(\dagger)$. They necessarily satisfy
  $r_{\sigma, \sigma'} \geq 0$ for all adjacent $\sigma, \sigma'$
  (i.e., they are ``positive orthogonal sets''), but the $r_{\sigma,
    \sigma'}$ need not be positive (i.e., $\{\mu_\sigma\}$ need not be
  strictly positive, as in the ample case). The associated divisors
  are in particular a nonnegative linear combination of prime
  $T$-invariant divisors.}  In what follows we will primarily be
interested in special ample polytopes.  Note that the Weil divisors of
such ample polytopes in particular must have strictly positive
coefficients of all prime $T$-invariant divisors (but this condition
does not imply speciality).

\subsection{Statement of main results}

Our first main result is the following, which will be proved in \S
\ref{s:main-proof}:

\begin{thm}\label{main}
  Let $P$ be a special ample polytope as above and let $m \in
  \Z_{>0}$. Consider the dilated polytope $m P:=\{mx : x \in P
  \}$. Then any point $z \in \Lambda(mP)$ can be written as a sum
  $z=z_1+ \cdots + z_m$, with $z_i \in \Lambda(P)$, $\forall i
  =1,\ldots, m$.
\end{thm}

The toric interpretation of the theorem is as follows.  Call an
equivariant line bundle $\mathcal{L}$ on $V$ \emph{special ample} if
it corresponds to a special ample polytope $P$.

\begin{cor}\label{toricmain}
  Let $\mathcal{L}$ be a special ample
  line bundle on $V$.  Then, the canonical
  map $$H^0(V, \mathcal{L}) \otimes H^0(V, \mathcal{L}) \otimes \cdots
  \otimes H^0(V, \mathcal{L}) \longrightarrow H^0(V, \mathcal{L}^m)$$
  is a surjection for all $m \geq 1$, i.e., $\mathcal{L}$ is projectively
  normal.
\end{cor}

\begin{rem}
  The above corollary is a special case of Oda's Conjecture which
  claims that the statement of the corollary is true for any ample
  line bundle on a nonsingular, projective toric variety. In the case
  of root systems of type $A$, the conjecture, and therefore the
  corollary, is known to be true (see \cite{howard}).
\end{rem}

Next, consider the semigroup $S_{P} \subset X \times \Z$ generated by
$(x, 1)$ for $x \in \Lambda(P)$.  Then, the main theorem is equivalent
to the statement that $S_{P}$ is normal, i.e., it is saturated in $X
\times Z$. In other words, it equals its saturation, $\overline{S_{P}}
:= \bigcup_{m \geq 1} (\Lambda(mP)) \times \{m\}$, i.e., the
intersection of the cone $\R_{> 0} \cdot (P \times \{1\})$ with the
lattice $\{(y + t\mu_\sigma, t) \mid y \in Y, t \in \Z\}$, for any
fixed $\sigma \in F_n$.

If we instead begin with the semigroup $\overline{S_{P}}$, then
Theorem \ref{main} is equivalent to the statement that this semigroup
is generated in degree one with respect to the grading $|(x,m)|=m$,
for $x \in X$ and $m \in \Z$.

Our second main result is
\begin{thm} \label{main2} The semigroup $S_{P} = \overline{S_{P}}$ is
  presented by quadratic relations.  In other words, $S_{P} = \langle
  \Lambda(P) \times \{1\} \rangle / (R)$, where $R$ is spanned by the
  elements
\begin{equation*}
(x,1) (y,1) - (x', 1) (y', 1),
\end{equation*}
for $x, y, x', y' \in \Lambda(P)$ such that $x+y = x'+y'$.
\end{thm}

\begin{rem} \label{r:quad-stronger} Put differently, the 
  semigroup ring $\C[S_{P}] = \C[\overline{S_{P}}]$ is quadratic. We
  will actually prove a stronger version of the above theorem, which
  roughly says that $(R)$ is spanned by moves which replace $(x_1,
  \ldots, x_m) \in \Lambda(P)^m$ by $(x_1, \ldots, x_i + \alpha,
  x_{i+1} - \alpha, \ldots, x_m)$ for $\alpha \in \Delta$. See \S
  \ref{main2ssec} below for a precise statement.
\end{rem}

Since $\C[\overline{S_{P}}] \cong \bigoplus_{m \geq 0} H^0(V,
\mathcal{L}^{m})$, the toric interpretation of Theorems \ref{main} and
\ref{main2} is
\begin{cor} \label{toricmain2} Let $\mathcal{L}$ be a special
ample line bundle on $V$.
 Then, the ring $\bigoplus_{m \geq 0} H^0(V,
  \mathcal{L}^{m})$ is quadratic.
\end{cor}
\begin{rem} The above corollary is a special case of Sturmfels's
  conjecture \cite[Conjecture 13.19]{Stugbcp}, which states that, for
  any projective nonsingular toric variety $X$ and ample projectively
  normal line bundle $\mathcal{L}$, the associated ring $\bigoplus_{m
    \geq 0} H^0(X, \mathcal{L}^{m})$ is quadratic. (If Oda's conjecture is
  true, then the projectively normal assumption is automatic.)
\end{rem}

This leaves open the natural
\begin{ques} Is the ring $\C[S_{P}]$ Koszul? 
\end{ques}
See \cite{payne} and \S \ref{ss:diag-split-intro} and
\S \ref{s:not-diag-split} below for such a result in a related situation.

\begin{rem}
  It is clear that all of the above results remain true if we replace
  $P$ with the polytope $\nu + P$, where $\nu \in X$ (we still require
  that $P$ satisfy $(\dagger)$). Concerning the geometric statements
  (Corollaries \ref{toricmain} and \ref{toricmain2}), the line bundle
  on $V$ corresponding to the polytope $\nu + P$ is isomorphic to the
  line bundle $\mathcal{L}$ (just equipped with a different
  equivariant structure, which does not affect these statements).  In
  other words, the above results can be viewed as applying to
  nonequivariant ample line bundles which admit a special equivariant
  structure.
\end{rem}

\subsection{Strengthening Theorems \ref{main} and
  \ref{main2}} \label{main2ssec}

Rather than prove Theorem \ref{main2}, we will prove the following,
which generalizes it and Theorem \ref{main}. For yet another
strengthening, see the appendix.

\begin{defn}\label{equiv}
  Suppose $P_1, \ldots, P_m$ are special ample polytopes and $(x_1,
  \ldots x_m) \in \Lambda(P_1) \times \cdots \times \Lambda(P_m)$.
  Suppose further that $\beta \in \Delta$ is a root and $i$ and $j$
  are indices such that $x_i + \beta \in \Lambda(P_i)$ and $x_j -
  \beta \in \Lambda(P_j)$. Then, we say that
  \begin{equation}\label{eq:equivmove}
    (x_1, \ldots, x_m) \sim (x_1, \ldots, x_{i-1}, x_i + \beta, x_{i+1}, 
    \ldots, x_{j-1}, x_j - \beta, x_{j+1}, \ldots, x_m).
\end{equation}
Call this a \emph{root move}.
Extend $\sim$ to the equivalence relation generated by this, i.e.,
$(x_1, \ldots, x_m) \sim (x_1', \ldots, x_m')$ if the two are related
by a sequence of root moves.
\end{defn}
Note that, since root moves are reversible, a tuple is related to
another tuple by root moves if and only if one can be obtained from
the other by a sequence of root moves.

The following result strengthens Theorems \ref{main} and \ref{main2}:
\begin{thm}\label{main2s}
  If $P_1, \ldots, P_m$ are special ample polytopes and $x \in
  \Lambda(P_1 + \cdots + P_m)$, then
\begin{enumerate}
\item[(i)] There exists a tuple $(x_1, \ldots, x_m) \in \Lambda(P_1)
  \times \cdots \times \Lambda(P_m)$ such that $x_1 + \cdots + x_m =
  x$;
\item[(ii)] If $(x_1, \ldots, x_m)$ and $(x_1', \ldots, x_m')$ are two
  such tuples, then $(x_1, \ldots, x_m) \sim (x_1', \ldots, x_m')$.
\end{enumerate}
\end{thm}
Specializing to the case $m = 2$ and $P = P_1 = P_2$, part (ii)
implies that the permutation $(x_1, x_2) \mapsto (x_2, x_1)$ is a
series of root moves inside $\Lambda(P)^2$.  Therefore, in the case $P
= P_1 = \cdots = P_m$ for arbitrary $m$, the relation $\sim$ is
actually generated by root moves \eqref{eq:equivmove} with $j =
i+1$. This explains Remark \ref{r:quad-stronger}, and hence Theorem
\ref{main2s}.(ii) implies Theorem \ref{main2}.

Our motivation for allowing $P_1, \ldots, P_m$ to be distinct
polytopes is that it allows one to inductively prove the theorem on
$m$: one deduces the result for $m > 2$ from the pair $(P_1 + \cdots +
P_{m-1}, P_m)$.

A toric interpretation of part of the theorem is as follows.
Let $\mathcal{L}_1, \ldots, \mathcal{L}_m$ be special ample line bundles
 on $V$ and
 $$\varphi_{\mathcal{L}_1, \ldots, \mathcal{L}_m}: H^0(V,
  \mathcal{L}_1) \otimes \cdots \otimes H^0(V, \mathcal{L}_m)
  \longrightarrow H^0(V, \mathcal{L}_1 \otimes \cdots \otimes
  \mathcal{L}_m)$$
be the canonical map.
\begin{cor}\label{toricmain2s}
\begin{enumerate}
\item[(i)] $\varphi_{\mathcal{L}_1, \ldots,
    \mathcal{L}_m}$ is surjective.

\item[(ii)] The kernel of $\varphi_{\mathcal{L}_1, \ldots,
    \mathcal{L}_m}$ is spanned by the canonical subspaces
  $$\ker(\varphi_{\mathcal{L}_i, \mathcal{L}_j}) \otimes \bigotimes_{k
    \notin \{i,j\}} H^0(V, \mathcal{L}_k) \subseteq
  \ker(\varphi_{\mathcal{L}_1, \ldots, \mathcal{L}_m}).$$
\end{enumerate}
\end{cor}
Similarly, we can apply this to the Cayley sum polytope of polytopes $P_1, \ldots, P_k$.  Recall that this is defined as the polytope inside $(X \otimes_\Z \R) \times \R^{k}$ which is the convex hull of $(P_1 \times \{e_1\}) \cup \cdots \cup
(P_k \times \{e_k\})$, where $e_1, \ldots, e_k$ are the standard basis of $\R^k$. The resulting polytope is denoted by $P_1 * P_2 * \cdots * P_k$ and is considered with respect to the lattice $Y \times \Z^k$.
\begin{cor}\footnote{Thanks to S. Payne for observing this corollary.}
  Let $P_1, \ldots, P_k$ be special ample polytopes.  Then, the Cayley
  sum polytope $P = P_1 * \cdots * P_k$ is normal, and $\C[S_P]$ is
  quadratic.
\end{cor}
The corollary follows from Theorem \ref{main2s} as follows: for every
$m_1, \ldots, m_k \geq 0$, apply the theorem to the product $\Lambda(P_1)^{m_1}
\times \cdots \times \Lambda(P_k)^{m_k}$, with $m = m_1 + \cdots + m_k$. Note
here that the (degree-one) generators of $\C[S_P]$ are the elements
$((y,e_i),1) \in (\Lambda(P_i) \times \Z^k) \times \Z$, where $1 \leq
i \leq k$.

Finally, we give the toric interpretation of the corollary.  Let
$\mathcal{L}_1, \ldots, \mathcal{L}_k$ be special ample line bundles
on $V$.  Given a vector bundle $\mathcal{U}$, let
$\Sym^m(\mathcal{U})$ denote its $m$-th symmetric power.
\begin{cor}
  The ring $\bigoplus_{m \geq 0} H^0(V, \Sym^m(\mathcal{L}_1 \oplus
  \cdots \oplus \mathcal{L}_k))$ is quadratic.
\end{cor}

\subsection{Diagonal splitness}\label{ss:diag-split-intro}
A closely related toric variety to $V$, studied in, e.g.,
\cite{payne}, is the one whose fan is such that its rays (i.e.,
one-dimensional cones) are generated by the elements of $\Delta$: so,
its initial lattice is $Y$, dual to the initial lattice of $V$.
Denote this variety by $U$.

Suppose that $Q$ is an orthogonal polytope corresponding to the fan
associated to $U$, i.e., one which describes an equivariant line
bundle on $U$. Then, the main result of \emph{op.~cit.} was that, in
the case that the root system is of type $A, B, C$, or $D$, the
semigroup ring $\C[S_Q]$ is Koszul (and in particular, $S_Q$ is
normal and $\C[S_Q]$ is quadratic).  This was proved by showing that
$Q$ is always diagonally split (see, e.g., \emph{op.~cit.}, or \S
\ref{s:not-diag-split} below).  Note that, unlike in the present
paper, the arguments did not extend to the exceptional types $E, F$,
or $G$, and the speciality and ampleness assumptions were not
required.

  In contrast, in \S \ref{s:not-diag-split} below, we show that,
  except in the cases $A_1, A_2, A_3$, and $B_2$, the ample polytopes
  $P$ associated to the toric variety $V$ considered in this paper are
  not diagonally split, and therefore the above argument cannot be
  applied in our case for \emph{any} root systems other than these
  four.

\subsection{Organization of the paper}

In \S \ref{s:main-proof} we prove Theorem \ref{main}, where a crucial
step involves using a lemma of Stembridge (\cite[Cor. 2.7]{stem2})
stating that in the usual partial order of $\sigma$-dominant weights,
a weight $\nu$ covers another one $\nu'$ if and only if the difference
$\nu - \nu'$ is a root that is positive with respect to $\sigma$. In
\S \ref{s:numbers-game} we recall the numbers game with a cutoff (from
\cite{qendrimthesis}; see also \cite{GS}), which gives a useful
language to prove Theorem \ref{main2s}. The proof of the theorem is
then given in \S$\!$\S \ref{s:main2s-proof} and \ref{s:lemmaproofs}.
Note that one of our auxiliary results (Lemma \ref{m2lem1})
generalizes the above lemma of Stembridge. In \S
\ref{s:not-diag-split} we show that ample polytopes for the toric
varieties $V$ as above are not diagonally split (with the exception of
the cases when the root system $\Phi$ is of type $A_1, A_2, A_3$, or
$B_2$).

In the appendix, we give a generalization of Theorem \ref{main2s} in
terms of the numbers game: these allow one to restrict the type of
tuples needed in the equivalence $\sim$ above.

\subsection{Acknowledgements}
We thank Sam Payne for helpful comments.  The first author was
supported by an EPDI Fellowship. The second author is an AIM Five-Year
Fellow, and was partially supported by the ARRA-funded NSF grant
DMS-0900233. We thank IHES and MIT for hospitality.

\section{Proof of Theorem \ref{main}}\label{s:main-proof}

Recall that for a cone $\sigma \in F_n$, we denote by $\{ \alpha_{i,
  \sigma}: i \in I\} \subset \Delta$ the corresponding set of simple
roots. For a root $\gamma \in \Delta$, we say that it is positive or
negative with respect to the chamber $\sigma$ if $\gamma$ can be
written as a nonnegative or nonpositive linear combination of the
elements of $\{ \alpha_{i, \sigma}: i \in I\}$, respectively. We write
$D_{\sigma}$ for the set of $\sigma$-dominant elements of $X
\otimes_{\Z} \R$.

Note that $P$ is the convex hull of the points $\{ \mu_{\sigma} \in X:
\sigma \in F_n \}$. The next two lemmas will allow us to better
understand the shape of the polytope $P$. The first one uses the fact
that $P$ is ample and the second one that $P$ is
special.
\begin{lemma}\label{one}
\emph{(}see, e.g., \cite[Lemma 12.1, p.~445]{akot}\emph{)}
$$P = \bigcap_{\sigma \in F_n} C^*_{\sigma},$$ where $C^*_{\sigma} : =
\{ \mu_{\sigma} - \sum_{i=1}^{n} t_i \alpha_{i, \sigma} : t_i \in
\R_{\geq 0} \}$.
\end{lemma}

\begin{lemma}\label{two}
  \emph{(}see, e.g., \cite[Lemma 12.2, p.~445]{akot}\emph{)}  $$P
  \cap D_{\sigma} = C^*_{\sigma} \cap D_{\sigma}.$$
\end{lemma}
Specializing to the points in $\Lambda(P)$,  we obtain
\begin{equation}\label{e:ltwo-int}
\Lambda(P) \cap D_{\sigma} = \{
\nu \in D_{\sigma} \cap X : \nu \stackrel{\sigma}{\preceq}
\mu_{\sigma}\},
\end{equation}
where $\stackrel{\sigma}{\preceq}$ stands for the
partial order in $X$ determined by the chamber $\sigma$, i.e., $\nu
\stackrel{\sigma}{\preceq} \mu_{\sigma}$ if $\mu_{\sigma} - \nu$ is a
nonnegative integral linear combination of the roots $\{ \alpha_{i,
  \sigma}: i \in I \}$.

Fix a chamber $\sigma \in F_{n}$. Since $W$ acts simply
transitively on $\{ D_{\tau}$: $\tau \in F_n \}$ and since
$\Lambda(mP) = \bigcup_{w \in W} (\Lambda(mP) \cap (w D_{\sigma}))$,
it suffices to prove the statement of Theorem \ref{main} for $z \in
\Lambda(m P) \cap D_{\sigma} $.

By \eqref{e:ltwo-int}, every element $z
\in \Lambda(m P) \cap D_{\sigma}$ satisfies $z
\stackrel{\sigma}{\preceq} m \mu_{\sigma}$. Clearly, for $z = m
\mu_{\sigma}$ the assertion of Theorem \ref{main} is true. So, to
prove the theorem, it suffices to show that, whenever it holds for $x
\in \Lambda(m P) \cap D_{\sigma}$, then it also holds for every $z
\in \Lambda(m P) \cap D_{\sigma}$ such that $x$ covers $z$.  Here,
$x$ \emph{covers} $z$ means that $z \stackrel{\sigma}{\preceq} t
\stackrel{\sigma}{\preceq} x$ and $t \in D_{\sigma} \cap X$
implies that $t = z$ or $t = x$.

So, assume that the statement of the theorem is true for $x \in
\Lambda(m P) \cap D_{\sigma}$ and that $x$ covers $z \in \Lambda(m P)
\cap D_{\sigma}$. By a lemma of Stembridge (\cite[Cor. 2.7]{stem2};
see also \cite[Lemma 2.3]{rapoport} and Remark \ref{strem} below),
there exists a $\sigma$-positive root $\beta$ such that $x- z =
\beta$. Since $z$ is $\sigma$-dominant and $\beta$ is
$\sigma$-positive, $\langle z ,\beta^{\vee} \rangle \geq 0$, and thus
$\langle x -\beta ,\beta^{\vee} \rangle \geq 0$, i.e., $$\langle x
,\beta^{\vee} \rangle \geq 2.$$

By assumption, $x$ can be written as a sum $x=x_1 +
\cdots + x_m$, with $x_i \in \Lambda(P)$, $\forall i =1,\ldots,
m$. The last inequality guarantees that $\langle x_j ,\beta^{\vee}
\rangle \geq 1$ for at least one $j \in \{1,\ldots,m\}$. The
proposition below then ensures that $x_j-\beta \in P$:

\begin{prop}\label{mainclaim}
  Let $y \in \Lambda(P)$ and $\beta \in \Delta$. If $\langle y,
  \beta^{\vee} \rangle \geq 1$, then $y-\beta \in \Lambda(P)$.
\end{prop}

Now we put $z_i = x_i, \forall i \neq j$, and $z_j = x_j - \beta$, and
then $z = z_1 + \cdots + z_m$, which verifies the theorem.  This
concludes the proof of Theorem \ref{main}, and it only remains to prove
the last proposition.

\subsection{Proof of Proposition \ref{mainclaim}}

Let $y \in \Lambda(P)$ and $\beta \in \Delta$ be
such that $\langle y ,\beta^{\vee} \rangle \geq 1$.

We must show that, for every $\sigma \in F_n$, $y-\beta
\stackrel{\sigma}{\preceq} \mu_{\sigma}$. Note that, since $y \in
\Lambda(P)$, there exist nonnegative integers $h_{i, \sigma}$,
$i=1,\ldots, n$, such that
\begin{equation}\label{h}
\mu_{\sigma} - y = \sum_{i=1}^{n} h_{i, \sigma} \alpha_{i, \sigma}.
\end{equation}
Let $\beta = \sum_{i=1}^{n} b_{i, \sigma} \alpha_{i, \sigma}$. Then
$\mu_{\sigma} - (y - \beta) = \sum_{i=1}^{n} (h_{i, \sigma} + b_{i,
  \sigma}) \alpha_{i, \sigma}$. If the chamber $\sigma$ is such that
$\beta$ is positive with respect to it, then clearly $y-\beta
\stackrel{\sigma}{\preceq} \mu_{\sigma}$.

We are therefore left to consider only the chambers with respect to
which $\beta$ is negative. Denote the set of such chambers by
$F_{-}$. Then we can write $F_{-}$ as a disjoint union $F_{-} = F_{-}'
\cup F_{-}''$, where $$F_{-}' = \{ \tau \in F_{-} : \tau \, \text{is
  adjacent to a chamber with respect to which}\, \beta\, \text{is
  positive}\},$$ and $F_{-}'' = F_{-} \setminus F_{-}'$.

Since $P$ is special, $\langle \mu_{\sigma}, \beta^{\vee} \rangle \leq
0, \forall \sigma \in F_{-}$. Moreover, we claim that $\langle
\mu_{\sigma} , \beta^{\vee} \rangle \leq -1, \forall \sigma \in
F_{-}''$. This follows from the previous statement because $P$ is
ample: indeed, if $\sigma \in F_-''$, then by definition $-\beta$ is
positive but not simple for $\sigma$, so there exists $\alpha_{i,
  \sigma}$ (necessarily not equal to $-\beta$) such that $\langle
\alpha_{i, \sigma}, -\beta^\vee \rangle \geq 1$.  Therefore, if
$\langle \mu_\sigma, \beta^\vee \rangle = 0$, then if $\sigma'$ is the
chamber adjacent to $\sigma$ corresponding to $\alpha_{i, \sigma}$, it
follows that $\langle \mu_{\sigma'}, \beta^\vee \rangle = \langle
\mu_\sigma-r_{\sigma, \sigma'} \alpha_{i, \sigma}, \beta^\vee \rangle
\geq r_{\sigma,\sigma'} > 0$.  However, $\sigma' \in F_-$ by
definition, which furnishes a contradiction.

To prove the proposition, we claim that it suffices to show
that $y-\beta \stackrel{\sigma}{\preceq} \mu_{\sigma}$ when $\sigma
\in F_{-}'$. Since $\langle y - \beta, \beta^{\vee}\rangle \geq -1$ by
assumption, it will then follow that $y-\beta$ lies in all of the
half-spaces whose intersection defines $P$ (whose boundaries are
maximal-dimensional facets of $P$), except possibly for those whose
boundary planes meet vertices of $P$ only in $F_-''$. Suppose, for
sake of contradiction, that $y-\beta \notin P$.  Let $0 \leq t < 1$ be
maximal such that $y-t\beta \in P$. Then $y-t\beta$ lies on a boundary
plane meeting vertices of $P$ only in $F_-''$.  Since $\langle
\mu_{\sigma}, \beta^{\vee} \rangle \leq -1, \forall \sigma \in
F_{-}''$, it follows that $\langle y-\beta, \beta^\vee \rangle <
\langle y-t\beta, \beta^\vee \rangle \leq -1$. This is impossible,
since $\langle y, \beta^\vee \rangle \geq 1$.  Thus, $y-\beta \in P$,
as desired.

Thus, take $\sigma \in F_{-}'$. In the remainder of the proof, we show
that $y-\beta \stackrel{\sigma}{\preceq} \mu_{\sigma}$.  
Denote by $\tau$ a chamber in
$F$ that is adjacent to $\sigma$ and such that $\beta$ is positive
with respect to $\tau$. We write $\alpha_i$ instead of $\alpha_{i,
  \tau}$ for the simple roots corresponding to $\tau$. Since $\beta$
is negative with respect to $\sigma$, there exists $j \in I$ such
that $\beta = \alpha_{j} = -\alpha_{j,\sigma}$. Moreover,
\begin{equation}\label{alpha}
\alpha_{i} = \alpha_{i,\sigma} + \langle \alpha_i , \beta^{\vee} \rangle \beta.
\end{equation}

Since $P$ is ample, $\mu_{\sigma} = - r_{\tau,\sigma} \beta + \mu_{\tau}$,
for $r_{\tau, \sigma} > 0$.
Thus, using (\ref{h}) and applying (\ref{alpha}), we get 
$$\sum_{i=1}^{n} h_{i, \sigma}\alpha_{i,\sigma}
= - r_{\tau, \sigma} \beta + \sum_{i=1}^{n} h_{i, \tau}
(\alpha_{i,\sigma} + \langle \alpha_i , \beta^{\vee} \rangle \beta).$$
Since $\{\alpha_{i,\sigma}: i \in I\}$ is a basis for $\Delta$ and
$\alpha_{j,\sigma}=-\beta$, from the last identity we deduce that
\begin{gather}
  h_{i, \sigma} = h_{i, \tau}, \forall i \in I \setminus \{\,j\}, 
\quad \text{and} \notag \\
\label{hqk}
h_{j, \sigma} = r_{\tau, \sigma} - h_{j, \tau} - \sum_{i \in I
  \setminus \{\,j \}} h_{i, \tau} \langle \alpha_i, \beta^{\vee}
\rangle.
\end{gather}

Now, $\mu_{\sigma} - (y-\beta) = \left( \sum_{i=1}^{n} h_{i, \sigma}
  \alpha_{i, \sigma} \right) - \alpha_{j, \sigma}$, so in order to
prove that $y-\beta \stackrel{\sigma}{\preceq} \mu_{\sigma}$, it
suffices to show that $h_{j, \sigma} \geq 1$. For a contradiction,
assume $h_{j, \sigma} = 0$.  From (\ref{hqk}) we then get
\begin{equation}\label{q1}
  h_{j,\tau} - r_{\tau, \sigma} = - \sum_{i \in I \setminus \{\,j \}} 
  h_{i, \tau} \langle \alpha_i, \beta^{\vee} \rangle.
\end{equation}
Next, recall that $\langle y, \beta^{\vee} \rangle \geq 1$, so
\begin{multline*}
  \langle \mu_{\sigma}, \beta^{\vee} \rangle = \langle \mu_{\tau} -
  r_{\tau, \sigma} \beta, \beta^{\vee} \rangle = \langle y +
  \sum_{i=1}^{n}
  h_{i, \tau} \alpha_{i}\, ,\beta^{\vee} \rangle - 2r_{\tau, \sigma} \\
  \geq 1 + 2h_{j, \tau} - 2r_{\tau, \sigma} + \sum_{i \in I \setminus
    \{\,j \}} h_{i, \tau} \langle \alpha_i, \beta^{\vee} \rangle = 1 -
  \sum_{i \in I \setminus \{\,j \}} h_{i, \tau} \langle \alpha_i,
  \beta^{\vee} \rangle,
\end{multline*}
where to get the last equality we used (\ref{q1}). But, the last
expression is strictly positive (since $\langle \alpha_i, \beta^\vee
\rangle = \langle \alpha_i, \alpha_{j}^\vee \rangle \leq 0$ for all $i
\neq j$), and, since the polytope $P$ is special, $\langle
\mu_{\sigma}, \alpha_{j}^{\vee} \rangle \leq 0$, a contradiction. This
ends the proof that $y-\beta \stackrel{\sigma}{\preceq} \mu_{\sigma}$
and concludes the proof of Proposition \ref{mainclaim}.

\section{The numbers game with a cutoff}\label{s:numbers-game}
In order to prove Theorem \ref{main2}, we use the language of
\emph{the numbers game with a cutoff}, from \cite{qendrim3} (see also
\cite{GS}). In this section we recall what we will need.

\subsection{The usual numbers game}
We first recall Mozes's numbers game \cite{mozes}, which has been
widely studied (e.g., in \cite{Pro-bru, Pro-min, DE, Erik-no1,
  Erik-no2, erikconf, Erik-no3, Erik-no4, eriksson, Wild-no1,
  Wild-no2}). Fix an unoriented, finite graph with no loops and no
multiple edges. Let $I$ be the set of vertices. Fix also a Cartan
matrix $C = (c_{ij})_{i,j \in I} \in \R^I \times \R^I$, such that
$c_{ii} = 2$ for all $i$, $c_{ij} = 0$ whenever $i$ and $j$ are not
adjacent, and otherwise $c_{ij}, c_{ji} < 0$, and either $c_{ij}
c_{ji} = 4 \cos^2(\frac{\pi}{n_{ij}})$ (when $n_{ij}$ is finite) or
$c_{ij} c_{ji} \geq 4$ (when $n_{ij} = \infty$).

We will only need to consider the case where our graph is the
underlying graph of a Dynkin diagram (undirected and without multiple
edges), and $C$ is the standard Cartan matrix for the diagram, i.e.,
$c_{ij} = \langle \alpha_i, \alpha_j^\vee \rangle$. In particular,
$c_{ij} \in \Z$ for all $i, j$.  Hence, the reader may assume this if
desired.

The \emph{configurations} of the game consist of vectors from
$\R^I$. The moves of the game are as follows: For any vector ${v} \in
\R^I$ and any vertex $i \in I$ such that ${v}_i < 0$, one may perform
the following move, called \emph{firing the vertex $i$}: ${v}$ is
replaced by the new configuration $f_i({v})$, defined by
\begin{equation*}
f_i({v})_j = v_j - c_{ij} v_i.
\end{equation*}
The entries ${v}_i$ of the vector ${v}$ are called
\emph{amplitudes}. The game terminates if all the amplitudes are
nonnegative.  Let us emphasize that \emph{only negative-amplitude
  vertices may be fired}.\footnote{In some of the literature, the
  opposite convention is used, i.e., only positive-amplitude vertices
  may be fired.}

\begin{prop}\cite{eriksson}
  The numbers game is \emph{strongly convergent}: if the game can
  terminate, then it must terminate, and in exactly the same number of
  moves and arriving at the same configuration, regardless of the
  choices made.
\end{prop}

\subsection{The numbers game with a cutoff}

In \cite{qendrimthesis}, the numbers game \emph{with a cutoff} was
defined: The moves are the same as in the ordinary numbers game, but
the game continues (and in fact starts) only as long as all amplitudes
remain greater than or equal to $-1$.  Such configurations are called
\emph{allowed}. Every configuration which does not have this property
is called \emph{forbidden}, and upon reaching such a configuration the
game terminates (we lose). We call a configuration \emph{winning} if
it is possible, by playing the numbers game with a cutoff, to reach a
configuration with all nonnegative amplitudes.

In \cite{GS}, a simple criterion was given to determine when the
numbers game with a cutoff is winning.  We will restrict to the Dynkin
case, with $C$ the standard Cartan matrix. Let $\Delta$ be the set of
roots. Pick simple roots corresponding to the vertices of the Dynkin
diagram, and write $\Delta = \Delta_+ \sqcup -\Delta_+$, where
$\Delta_+$ is the set of positive roots.

We can view $\Delta \subseteq \Z^I$ and $\Delta_+ \subseteq \Z_{\geq
  0}^I$.  For $i \in I$, let $\alpha_i$ be the simple root
corresponding to $i$, which as an element of $\Z^I$ is the elementary
vector $(\alpha_i)_j = \delta_{ij}$. Note that, since $\alpha_i$
refers to a vector in $\Z^I$, in the case that $\alpha \in \Z^I$, we
will \emph{never} use $\alpha_i$ to refer to a component of $\alpha$,
reserving it exclusively for the elementary vector $\alpha_i \in
\Z^I$.

A useful description of $\Delta$ is
\begin{equation*}
\Delta = \bigcup_{i \in I} W \cdot \alpha_i,
\end{equation*}
where $W$ is the Weyl group generated by the simple reflections $s_i:
\R^I \rightarrow \R^I$, for all $i \in I$, given by
\begin{equation*}
s_i(\beta)_j =
\begin{cases} \beta_j, & \text{if $j \neq i$}, \\
  -\beta_i - \sum_{k \neq i} c_{ki} \beta_k, & \text{if $j = i$}.
\end{cases}
\end{equation*}

\begin{prop}\label{gsprop}\cite[Theorem 3.1, Corollary 5.10.(a)]{GS}
  Fix a Dynkin diagram with standard Cartan matrix $C$. Beginning with
  a configuration $v \in \R^I$, the numbers game with a cutoff is
  winning if and only if
\begin{equation} \label{gscond}
v \cdot \alpha \geq -1, \forall \alpha \in \Delta_+,
\end{equation}
and in this case, one always wins the numbers game with a cutoff, no
matter which moves are made, and arrives at the same final
configuration in the same total number of moves.
\end{prop}
Here, $\cdot$ is the dot product of $v \in \R^I$ with $\alpha \in
\Z^I$, i.e., $v \cdot \bigl(\sum_i c_i \alpha_i\bigr) = \sum_i c_i
v_i$, for $c_i \in \R$.

\subsection{Relation to the polytope $P$}

Proposition \ref{mainclaim} has the following consequence in terms of
the numbers game with a cutoff.  We consider the embedding
\begin{equation*}
  \iota: X \into \R^I, \quad x \mapsto \iota(x), 
  \quad \iota(x)_i := \langle x, \alpha_i^\vee \rangle.
\end{equation*}
In this language, the condition \eqref{gscond} translates for $x \in
X$ as follows: The configuration $\iota(x)$ is winning if and only if
\begin{equation}\label{gscond2}
\langle x, \alpha^\vee \rangle \geq -1, \forall \alpha \in \Delta_+.
\end{equation}
Then, Proposition \ref{mainclaim} implies
\begin{cor}\label{ngcor}
  If $x, y \in X$ and $\iota(y)$ can be obtained from $\iota(x)$ by
  playing the numbers game with a cutoff, then $x \in \Lambda(P)$ if
  and only if $y \in \Lambda(P)$.
\end{cor}
\begin{proof}
  Suppose that $u \in X$ and $\iota(u) \in \Z^I$ is obtained along the
  way from $\iota(x)$ to $\iota(y)$. From $u$, any move in the numbers
  game with a cutoff is of the form $u \mapsto u + \iota(\alpha_i)$
  for some $i \in I$ such that $u_i = -1$. Hence, $\langle u,
  \alpha_i^\vee \rangle = -1$ and $(u + \iota(\alpha_i))_i = \langle u
  + \alpha_i, \alpha_i^\vee \rangle = 1$.  We therefore conclude from
  Proposition \ref{mainclaim} that $u \in \Lambda(P)$ if and only if
  $u + \alpha_i \in \Lambda(P)$. The corollary follows.
\end{proof}
Note that the choice of simple roots was arbitrary, so the corollary
in fact holds for any choice of simple roots (equivalently, any choice
of dominant chamber).

\begin{rem}
The corollary extends to the case where $y$ is
obtained from $x$ in the usual numbers game by firing vertices only of
amplitude $-1$, i.e., we can continue the numbers game even if there
is an amplitude less than $-1$, as long as we never fire such
vertices. (This seems to be a reasonable variation on the numbers game
with a cutoff.)
\end{rem}

\section{Proof of Theorem \ref{main2s}}\label{s:main2s-proof}

It is convenient to abuse notation slightly, by omitting the map $\iota$:
\begin{ntn} If $x \in X$ and $\iota(x)$ is winning, we say also that
  $x$ is winning. Moreover, if $x, y \in X$ and $\iota(y)$ is obtained
  from $\iota(x)$ by playing the numbers game (with or without a
  cutoff), we also say that $y$ is obtained from $x$ by playing the
  numbers game (with or without a cutoff, respectively).
\end{ntn}

Fix once and for all a dominant chamber $\sigma$, and write $D$,
$\prec$, $\preceq$, and $\mu$, instead of $D_{\sigma}$,
$\stackrel{\sigma}{\prec}$, $\stackrel{\sigma}{\preceq}$, and
$\mu_{\sigma}$, respectively. We omit $\sigma$ from now on, and by a
dominant element we always mean an element of $D$.

Next, given special ample polytopes $P_1, \ldots, P_m$, we let $\mu_1,
\ldots, \mu_m$ denote the vertices $\mu_1, \ldots, \mu_m \in D$ of
each corresponding to the dominant chamber.

Finally, we recall the notion of \emph{length} of roots.  For
simply-laced root systems (i.e., types $A_n$, $D_n$, and $E_n$, since
we only consider the Dynkin case), we say that all roots have the same
length. For the other root systems, the set of roots $\Delta$ is
partitioned into the subsets of \emph{short} and \emph{long} roots,
and we say that the long roots are \emph{longer} than the short roots.
One way to define the partition (which will be useful to us) is that,
if $\beta \in \Delta$ is at least as long as $\alpha \in \Delta$ and
$\alpha \neq \pm \beta$, then $\langle \alpha, \beta^\vee \rangle \in
\{-1, 0, 1\}$. Recall also that the partition is preserved by the Weyl
group action.  We emphasize that, for us, $\langle \alpha, \alpha^\vee
\rangle = 2$ for all $\alpha \in \Delta$, long or short; the
terminology of length comes from the norm under the symmetrized Cartan
form, which we will not use.

\subsection{Outline of the proof}
First, Theorem \ref{main2s}.(i) follows in exactly the same manner as
Theorem \ref{main}. We give a short proof in the spirit of this
section, based on Proposition \ref{mainclaim}, in \S
\ref{ss:main2smgen} below.

Our strategy underlying the proof of Theorem \ref{main2s}.(ii) is to
perform induction on the sum $x_1 + \cdots + x_m$, which we can assume
is winning (in fact we could assume it is dominant using the action of
the Weyl group).  The induction will be over a certain partial order
on the sum polytope $P = P_1 + \cdots + P_m$.

The proof is broken into three parts: first we prove results about the
partial order on the winning locus of $P$, which boil down to a
strengthening of the lemma of Stembridge mentioned earlier.  Second,
we prove the theorem in the case $m=2$.  Third, we inductively deduce
the theorem for general $m$. In what follows, we will explain the
proof modulo some lemmas whose proofs will be provided in \S
\ref{s:lemmaproofs}.

\subsection{Partial ordering on the winning locus of $P$}
Let $P$ be a special ample polytope.
\begin{defn}
  Suppose $x \in \Lambda(P)$.  If $x \neq \mu$, then a simple root
  $\alpha$ is \emph{$P$-progressive} for $x$ if either $x$ is dominant
  and $\alpha$ has minimum length such that $x+\alpha \preceq \mu$, or
  else $\langle x, \alpha^\vee \rangle \leq -1$.
\end{defn}
It is immediate that, for all $x \neq \mu$, there exists a simple root
which is $P$-progressive for $x$.

This subsection is devoted to the proof of
\begin{prop} \label{p:powin} If $\alpha$ is $P$-progressive for $x$,
  then $x +\alpha \in \Lambda(P)$. Moreover, if $x$ is winning, then
  so is $x+\alpha$.
\end{prop}
\begin{proof}
  First, suppose that $\alpha$ is a simple
  root such that $\langle x, \alpha^\vee \rangle \leq-1$. By
  Proposition \ref{mainclaim}, $x+\alpha \in \Lambda(P)$.  If $x$ is
  winning, then $x+\alpha$ is obtained from $x$ by a move of the
  numbers game, and hence it is winning.

  If $x$ is dominant and $\alpha$ is a simple root of minimum length
  such that $x+\alpha \preceq \mu$, the result follows from the case
  $y=\mu$ of the Lemma \ref{m2lem1} below.  Namely, by Corollary
  \ref{ngcor}, to show that $x+\alpha \in \Lambda(P)$, it suffices to
  show that $z \in \Lambda(P)$, where $z$ is the result of playing the
  numbers game with a cutoff beginning with $x+\alpha$.  Next, if
  $\beta \in \Delta_+$ is any positive root such that $x+\beta \preceq
  \mu$, then $\beta$ must be at least as long as $\alpha$; otherwise
  $\beta$ would be short and $\alpha$ long, and there would exist a
  short simple root $\gamma$ such that $\gamma \preceq \beta$.  In the
  latter case, $x + \gamma \preceq x+\beta \preceq \mu$, contradicting
  our assumption that $\alpha$ has minimum length such that $x+\alpha
  \preceq \mu$.  Therefore, we may apply Lemma \ref{m2lem1} with $y =
  \mu$.  We conclude that $x+\alpha$ is winning, i.e., $z$ is
  dominant, and also $z \preceq \mu$.  Then, $z \in \Lambda(P)$ by
  \eqref{e:ltwo-int}.
\end{proof}
\begin{lemma}\label{m2lem1}
  Suppose $x \prec y$ and $x, y \in X \cap D$.  Let $\alpha \in
  \Delta_+$ be a positive root of minimum length such that $x + \alpha
  \preceq y$. Then, $x+\alpha$ is winning, and if $z$ is the result of
  playing the numbers game with a cutoff, then $x+\alpha \preceq z
  \preceq y$.
\end{lemma}
The lemma will be proved in \S \ref{ss:m2lem1-proof}.
\begin{rem} \label{strem} Lemma \ref{m2lem1} strengthens the
    aforementioned result of Stembridge.  Specifically, if $y$ covers
    $x$, then $y=z$, i.e., $y$ is obtainable from $x+\alpha$ by
    playing the numbers game with a cutoff.  In this case, $y = x +
    \beta$, where $\beta \in \Delta_+$ is obtained from $\alpha$ by
    playing the numbers game (using the same firing sequence as
    for $x+\alpha \mapsto x+\beta$, which involves firing only
    vertices of amplitude $-1$).  This was our motivation for
    replacing $\mu$ by $y$ in the statement of the lemma.
\end{rem}

\subsection{The case $m = 2$ of Theorem
  \ref{main2s}.(ii)} \label{ss:main2smeq2} The heart of the proof of
Theorem \ref{main2s}.(ii) is contained in the case $m=2$. Then,
general $m$ will be a straightforward generalization.  In turn, the
case $m=2$ is based on the following lemma.
\begin{lemma} \label{m2lem2} Let $P_1$ and $P_2$ be special ample
  polytopes, $(x_1, x_2) \in \Lambda(P_1) \times \Lambda(P_2)$, $P =
  P_1 + P_2$, and $x = x_1 + x_2 \in \Lambda(P)$.  If $\alpha$ is
  $P$-progressive for $x$, then there exists $(x_1', x_2')$ and an
  index $i\in\{1,2\}$ such that $(x_1, x_2) \sim (x_1', x_2')$ and
  $\alpha$ is $P_i$-progressive for $x_i'$.
\end{lemma}
This will be proved in \S \ref{ss:m2lem2-proof} below. Here, we explain
how it implies Theorem \ref{main2s}.(ii) in the case $m=2$.
\begin{proof}[Proof of Theorem \ref{main2s}.(ii) for $m=2$] As remarked
  earlier, it is enough to prove the theorem in the case that $x$ is
  winning. Let $\mu = \mu_1 + \mu_2$ (where by convention $\mu_i$ is
  the vertex of $P_i$ corresponding to the dominant chamber).  The
  theorem is immediate in the case that $x = \mu$, since then $x = x_1
  + x_2$ implies that $x_1 = \mu_1$ and $x_2 = \mu_2$ (and
  vice-versa).  Inductively, suppose that $x \in \Lambda(P)$ is
  winning, and for some $P$-progressive $\alpha$, the theorem holds
  for $x+\alpha$.


  Suppose that $(x_1, x_2), (x_1', x_2') \in \Lambda(P_1) \times
  \Lambda(P_2)$ are pairs such that $x_1 + x_2 = x = x_1'+x_2'$.  Let
  $\alpha$ be $P$-progressive for $x$. By Lemma \ref{m2lem2} (applied
  to both $(x_1, x_2)$ and $(x_1', x_2')$ separately), it is enough to
  assume that there exist indices $i$ and $j$ such that $\alpha$ is
  $P_i$-progressive for $x_i$ and $P_j$-progressive for $x_j'$.
  Without loss of generality, suppose that $i = 1$.  Let $(y_1, y_2)$
  and $(y_1', y_2')$ be given by $y_1 = x_1+\alpha$, $y_2 = x_2$,
  $y_j' = x_j' + \alpha$, and $y_k' = x_k'$ for $k \neq j$.  Since
  $x_1+x_2+\alpha = x + \alpha = x_1' + x_2'+\alpha$, by hypothesis,
  $(y_1, y_2) \sim (y_1', y_2')$.  By induction on the number of root
  moves \eqref{eq:equivmove} required to realize the latter
  equivalence, Lemma \ref{m2lem3} below then implies that either
  $(x_1,x_2) \sim (y_1'-\alpha, y_2') \in \Lambda(P_1) \times
  \Lambda(P_2)$ or $(x_1, x_2) \sim (y_1', y_2' - \alpha) \in
  \Lambda(P_1) \times \Lambda(P_2)$ (where, for the purposes of
  induction, we drop the assumption that $\alpha$ is $P_1$-progressive
  for $x_1$ and assume only that $x_1+\alpha \in \Lambda(P_1)$, and
  similarly for $x_j'$ and $P_j$). If the result is $(x_1', x_2')$, we
  are done. If not, the result must be $(x_1' \pm \alpha, x_2' \mp
  \alpha)$, which is related to $(x_1', x_2')$ by a single root move.
\end{proof}
Above we needed the following lemma, whose proof will be given in \S
\ref{ss:m2lem3-proof}.
\begin{lemma} \label{m2lem3} Suppose that $(x_1, x_2) \in \Lambda(P_1)
  \times \Lambda(P_2)$, and $\alpha$ is a simple root such that
  $x = x_1+x_2$ satisfies 
  $\langle x, \alpha^\vee \rangle \geq -1$, 
  and such that $x_1 + \alpha \in \Lambda(P_1)$.  If $\beta \in
  \Delta$ is such that $(x_1+\alpha + \beta, x_2-\beta) \in
  \Lambda(P_1) \times \Lambda(P_2)$, then either $(x_1 + \beta, x_2 -
  \beta) \in \Lambda(P_1) \times \Lambda(P_2)$ or $(x_1 + \alpha +
  \beta, x_2 - \alpha - \beta) \in \Lambda(P_1) \times \Lambda(P_2)$.
  In the latter case, either $\alpha + \beta \in \Delta$, or $(x_1 +
  \alpha, x_2 - \alpha) \in \Lambda(P_1) \times \Lambda(P_2)$.
\end{lemma}
\subsection{Proof of Theorem \ref{main2s} for general
  $m$}\label{ss:main2smgen}
Let $\mu = \mu_1 +
  \cdots + \mu_m$, where $\mu_i$ is the vertex of $P_i$ corresponding
  to the dominant chamber. The theorem is immediate in the case
  $x=\mu$.  It is enough to prove the theorem when $x$ is winning,
  under the inductive hypothesis that the theorem holds for $x +
  \alpha$ where $\alpha$ is $P$-progressive for $x$.

  To prove part (i), suppose that $(y_1, \ldots, y_m) \in \Lambda(P_1)
  \times \cdots \times \Lambda(P_m)$ with $y_1 + \cdots + y_m =
  x+\alpha$. Then for some index $i$, $\langle y_i, \alpha^\vee \rangle
  \geq 1$, and by Proposition \ref{mainclaim}, $(y_1, \ldots, y_{i-1},
  y_i-\alpha, y_{i+1}, \ldots, y_m) \in \Lambda(P_1) \times \cdots
  \times \Lambda(P_m)$, with the desired sum $x$.

  For part (ii), we will additionally induct on $m$, i.e., we assume
  that the theorem holds for smaller values of $m$.  Let $Q := P_1 +
  \cdots + P_{m-1}$, so that $P = Q + P_m$. Let $y = x_1 + \cdots +
  x_{m-1}$ and $y' = x_1' + \cdots + x_{m-1}'$.  Then, by the previous
  subsection, there exist root moves relating $(y, x_m)$ to $(y',
  x_m')$.  To turn this into root moves relating $(x_1, \ldots, x_m)$
  and $(x_1', \ldots, x_m')$, it is enough to apply the theorem for
  the case $m-1$ (i.e., for $(P_1, \ldots, P_{m-1})$) together with
  the following lemma.
  \begin{lemma} \label{l:main2s-indlem} Suppose that $y \in Q = P_1 +
    \cdots + P_{m-1}$, $\beta \in \Delta$, and $y + \beta \in Q$.
    Assume Theorem \ref{main2s}.(i) holds for $(P_1, \ldots,
    P_{m-1})$. Then, there exists a tuple $(y_1, \ldots, y_{m-1}) \in
    \Lambda(P_1) \times \cdots \times \Lambda(P_{m-1})$ such that $y =
    y_1 + \cdots + y_{m-1}$ and an index $j$ such that $y_j + \beta
    \in \Lambda(P_j)$.
\end{lemma}
The lemma will be proved in \S \ref{ss:l:main2s-indlem-proof}.

\section{Proof of lemmas} \label{s:lemmaproofs}
\subsection{Proof of Lemma \ref{m2lem1}} \label{ss:m2lem1-proof} We
will use the following general result:
\begin{claim}\label{baswincl}
  If $x \in \Z^I$ is dominant and $\alpha \in \Delta_+$ is any
  positive root, and the usual numbers game on $x+\alpha$ does not
  involve firing any vertices corresponding to simple roots shorter
  than $\alpha$, then $x+\alpha$ is winning.
\end{claim}
In particular, if $\alpha$ is a short positive root and $x \in \Z^I$
is dominant then $x+\alpha$ is winning (equivalently, all short
positive roots are themselves winning).
\begin{proof}
  Let us play the usual numbers game on $x + \alpha$.  If we fire a
  vertex $i$ corresponding to a simple root $\beta$ whose length is at
  least that of $\alpha$, then since $\langle \alpha, \beta^\vee
  \rangle \geq -1$, the amplitude at vertex $i$ is $-1$. Since the
  length of $\alpha + \beta$ is equal to that of $\alpha$, we can
  replace $\alpha$ with $\alpha+\beta$, and then $x+(\alpha+\beta)$
  takes one fewer move under the numbers game to reach a dominant
  configuration. By induction on the number of moves required to play
  the numbers game on $x + \alpha$, we see that all vertices fired
  have amplitude $-1$, and hence $x + \alpha$ is winning.
\end{proof}
Suppose that $y \in X \cap D$ and $\alpha$ is a positive root of
minimum length such that $x+\alpha \preceq y$.  Let us play the
numbers game with a cutoff on $x+\alpha$.  We claim that this only
involves firing vertices corresponding to simple roots of length at
least $\alpha$.  Then, by Claim \ref{baswincl}, $x+\alpha$ is winning.
Moreover, the result $z$ of playing the numbers game with a cutoff is
the dominant configuration obtained from $x+\alpha$ by adding the
minimum positive combination of simple roots.  Since $y$ is dominant
and $x+\alpha \preceq y$, $y$ is also such a configuration, and it
follows that $z \preceq y$.

It remains to show that playing the numbers game beginning with
$x+\alpha$ does not involve firing a vertex corresponding to a simple
root of length shorter than $\alpha$. For a contradiction, suppose
not, and consider the first vertex fired corresponding to a shorter
simple root. Call this simple root $\gamma$.  It follows as above that
every dominant configuration is obtainable from $x+\alpha$ by adding
simple roots adds $\gamma$ as well.  Therefore, since $x+ \alpha
\preceq y$ and $y$ is dominant, also $x +\alpha+\gamma \preceq y$ and
hence $x+\gamma \preceq y$, which is a contradiction.

\subsection{Proof of Lemma \ref{m2lem2}} \label{ss:m2lem2-proof}
First, if $x$ is not dominant, then $\langle x, \alpha^\vee \rangle
\leq -1$, and hence $\langle x_i, \alpha^\vee \rangle \leq -1$ for
some $i$, which shows that $\alpha$ is $P_i$-progressive for $x_i$.  So
we can restrict to the dominant case.  Thus, $\alpha$ has minimal
length among simple roots such that $x+\alpha \preceq \mu$.

Given a simple root $\alpha$, let $P^\alpha$ denote the
maximum-dimensional boundary facet of $P$ meeting $\mu$ which is
parallel to the span of all simple roots other than $\alpha$.  In
other words (using Lemma \ref{one}), $x \in \Lambda(P^\alpha)$ if and
only if $x \in \Lambda(P)$ but $x + \alpha \not \preceq \mu$.
\begin{claim}\label{m2lem2cl}
If any element $x_i$ of the pair $(x_1, x_2)$ is dominant,
then either $x_i \in \Lambda(P^\alpha_i)$, or else $\alpha$ is
$P_i$-progressive for $x_i$.
\end{claim}
\begin{proof}
  If $x_i \notin \Lambda(P^\alpha_i)$, then $\alpha$ must be of
  minimal length with this property, since $x_i \notin
  \Lambda(P^\beta_i)$ implies that $x_1 + x_2 \notin
  \Lambda(P^\beta)$, which implies by assumption that $\beta$ is at
  least as long as $\alpha$.
\end{proof}
Now, we prove the lemma. If, for any simple root $\beta$, $\langle
x_1, \beta^\vee \rangle \leq -1$ but $\langle x_2, \beta^\vee \rangle
\geq 1$, we can perform a move $(x_1, x_2) \mapsto (x_1 + \beta, x_2 -
\beta)$.  So, after performing such moves, we can assume that this
does not happen. Since $x = x_1 + x_2$ is dominant, this implies that
$x_1$ is dominant.  By Claim \ref{m2lem2cl}, we are done unless $x_1
\in P_1^\alpha$. So assume this is the case. By performing moves of
the form $(x_1, x_2) \mapsto (x_1 - \beta, x_2 + \beta)$ for simple
roots $\beta \neq \alpha$ (which may make $x_1$ no longer dominant,
but preserves the property that $x_1 \in P_1^\alpha$), we can assume
that $\langle x_2, \beta^\vee \rangle \geq 0$ for all simple roots
$\beta \neq \alpha$, without changing the assumption that $x_1 \in
P_1^\alpha$.  Then, either $\langle x_2, \alpha^\vee \rangle \leq -1$,
or $x_2$ is dominant. In the former case, $\alpha$ is
$P_2$-progressive for $x_2$, as desired. In the latter case, by Claim
\ref{m2lem2cl}, it is enough to suppose that $x_2 \in
P_2^\alpha$. However, in this case, $x=x_1 + x_2 \in P_1^\alpha +
P_2^\alpha = P^\alpha$, contradicting our assumption that $x + \alpha
\preceq \mu$.

\subsection{Proof of Lemma \ref{m2lem3}} \label{ss:m2lem3-proof}
First, if $\langle x_1 + \alpha + \beta, \alpha^\vee \rangle \geq 1$,
then $x_1 + \beta\in \Lambda(P_1)$ by Proposition \ref{mainclaim}.
Since $x_2 -\beta \in \Lambda(P_2)$ by assumption, this proves the
lemma.  Next, suppose that $\langle x_1 + \alpha + \beta, \alpha^\vee
\rangle \leq 0$, i.e., $\langle x_1 + \beta, \alpha^\vee \rangle \leq
-2$. Since $x = (x_1 + \beta) + (x_2 - \beta)$ satisfies
$\langle x, \alpha^\vee \rangle \geq -1$, it follows that $\langle
x_2 - \beta, \alpha^\vee \rangle \geq 1$.  By Proposition
\ref{mainclaim}, $x_2 - \beta - \alpha \in \Lambda(P_2)$. Since $x_1 +
\alpha + \beta \in \Lambda(P_1)$ by assumption, this proves that $(x_1
+ \alpha + \beta, x_2 - \beta - \alpha) \in \Lambda(P_1) \times
\Lambda(P_2)$. It remains to prove the final assertion. Suppose that
$(x_1 + \alpha, x_2 - \alpha) \notin \Lambda(P_1) \times
\Lambda(P_2)$. By assumption $x_1 + \alpha \in \Lambda(P_1)$, so $x_2
- \alpha \notin \Lambda(P_2)$.  In view of Proposition
\ref{mainclaim}, $\langle x_2, \alpha^{\vee}\rangle \leq 0$.  Since
$\langle x_2 - \beta, \alpha^{\vee} \rangle \geq 1$ (as observed
above), this implies that $\langle \beta, \alpha^\vee \rangle \leq
-1$. In this case, $\alpha + \beta$ must be a root.

\subsection{Proof of Lemma
  \ref{l:main2s-indlem}}\label{ss:l:main2s-indlem-proof}
First, consider the case that $\langle y, \beta^\vee \rangle <
0$. Then, we can let $(y_1, \ldots, y_{m-1}) \in \Lambda(P_1) \times
\cdots \times \Lambda(P_{m-1})$ be arbitrary such that $y = y_1 +
\cdots + y_{m-1}$, and then for some $j$ one must have $\langle y_j,
\beta^\vee \rangle < 0$ as well, so that $y_j + \beta \in
\Lambda(P_j)$ by Proposition \ref{mainclaim}. Similarly, if $\langle
y, \beta^\vee \rangle \geq 0$, then $\langle y + \beta, \beta^\vee
\rangle > 0$, and we can take any $(z_1, \ldots, z_{m-1}) \in
\Lambda(P_1) \times \cdots \times \Lambda(P_{m-1})$ such that $y +
\beta = z_1 + \cdots + z_{m-1}$. Then, there exists some $j$ such that
$\langle z_j, \beta^\vee \rangle > 0$, so again by Proposition
\ref{mainclaim}, $z_j - \beta \in \Lambda(P_j)$. Hence, the tuple
$(z_1, \ldots, z_{j-1}, z_j - \beta, z_{j+1}, \ldots, z_m)$ satisfies
the needed conditions.

\section{Ample polytopes are not diagonally split, after
  Payne}\label{s:not-diag-split}

As mentioned in the introduction, Payne (in \cite{payne}) considers a
toric variety, $U$, similar to the one we consider, $V$, but for which
the rays of the fan are $\R_{\geq 0} \cdot \alpha$, for all $\alpha
\in \Delta$.  He proves that, in types $A, B, C$, and $D$, for all
lattice polytopes $P$ corresponding to a torus-equivariant line bundle
on $U$ (even if not ample), the corresponding semigroup
$S_P$ is normal, and the ring $\C[S_P]$ is
Koszul.  This follows from the fact, that he proves, that such lattice
polytopes are \emph{diagonally split} for some integer $q \geq 2$.

Here, we show that ample polytopes for the varieties $V$ considered in
this paper are diagonally split for some integer $q \geq 2$ only in
the cases $A_1, A_2, A_3$, and $B_2 (= C_2)$.

Recall from \cite{payne} the following definition.  Let $\Gamma$ be a
lattice with dual lattice $\Gamma^\vee$, and let $\Gamma_\R := \Gamma
\otimes_\Z \R$ and $\Gamma^\vee_\R := \Gamma^\vee \otimes_\Z \R$. Let
$P \subseteq \Gamma \otimes_\Z \R$ be a lattice polytope (with
vertices in $\Gamma$). Let $v_1, \ldots, v_k \in \Gamma^\vee$ be the
primitive lattice generators of the inward normal rays of the facets
of $P$. Define

\begin{equation}
  \mathbb{F}_P^{\circ} := \{u \in \Gamma_\R  \mid -1 < \langle u, 
  v_i \rangle  < 1, \forall i =1, \ldots, k\}.
\end{equation}
Let $q \geq 2$ be an integer.  Then, $P$ is \emph{diagonally split}
for $q$ if and only if every element $z \in (\frac{1}{q}
\Gamma)/\Gamma$ has a representative $\tilde z \in
\mathbb{F}_P^{\circ} \cap \frac{1}{q} \Gamma$.

Note that, in our case, $\Gamma = Y$. It is clear that all lattice
polytopes corresponding to equivariant line bundles on a toric variety
are diagonally split if and only if the polytopes corresponding to
ample bundles are diagonally split. Moreover, such polytopes are
diagonally split if and only if any one such polytope is diagonally
split.

\begin{prop}
  An ample polytope \emph{(}in $Y_\R$, with vertices in
  $Y$\emph{)} is diagonally split for some $q \geq 2$ if and only if
  the root system is of type $A_1, A_2, A_3$, or $B_2 (= C_2)$.  For
  $A_1$ and $A_2$, ample polytopes are diagonally split for all
  $q \geq 2$, and for $A_3$ and $B_2$, ample polytopes are
  diagonally split for odd but not even $q \geq 2$.
\end{prop}

\begin{proof}
  The inward primitive normal vectors for an ample polytope
  are the images under the Weyl group of the fundamental coweights
  $\omega_i, i \in I$.  Hence, the polytope is diagonally split for
  $q$ if and only if, for all $z \in \frac{1}{q} Y / Y$, there is a
  representative $\tilde z \in \frac{1}{q} Y$ such that $-1 < \langle
  w \tilde z, \omega_i \rangle < 1$ for all $i \in I$ and all $w \in
  W$.

  We first prove that such polytopes are not diagonally split if the
  root system is not listed above. Such root systems contain, as a
  subsystem, either a root system of type $A_4$, $D_4$, $B_3$, $C_3$,
  or $G_2$.  It is clear that, for this direction, it suffices to show
  that, for every $q \geq 2$, ample polytopes for these
  four root systems are not diagonally split.  To do so, it suffices
  to exhibit in each of these cases a particular element $z \in
  \frac{1}{q} Y / Y$ such that, for all representatives $\tilde z \in
  \frac{1}{q} Y$, there exists $w \in W$ and $i \in I$ such that $|
  \langle w \tilde z, \omega_i \rangle | \geq 1$.

We use the standard labeling of roots as in \cite[\S VI.4]{bourbaki} (which we
will also recall). Also, for
every $i \in I$, we denote by $s_i$ the simple reflection
corresponding to the simple root $\alpha_i$.

First let us consider a root system of type $A_3$ and even $q$, and
show that $P$ is not diagonally split. Recall that, for $A_n$ type,
the simple roots $\alpha_1, \ldots, \alpha_n$ are linearly ordered
along a line segment. We consider the element $z := \frac{1}{2}
\alpha_1 + \frac{1}{2} \alpha_3$.  Then, for every $\tilde z \in z +
Y$, either $|\langle \tilde z, \omega_i \rangle| \geq 1$ for some $i$,
or $\tilde z$ is in the same Weyl orbit as $z$.  But, $\langle s_2 z,
\omega_2 \rangle = 1$, which yields the desired inequality.  In
particular, ample polytopes for any root system containing $A_3$ are
not diagonally split for even $q$. Also, the same is true for root
systems containing $B_3$ or $C_3$.  Thus, for the cases $A_4, D_4,
B_3$, and $C_3$, it suffices to restrict our attention to the case
where $q$ is odd.

From now on, fix an odd integer $q \geq 3$ and set $p :=
\frac{q-1}{2}$.  Suppose that the root system is of type $A_4$. Then,
we consider the element
\begin{equation}
  z := \frac{p+1}{q} \alpha_1 + \frac{p+1}{q} \alpha_3 + 
  \frac{1}{q} \alpha_4.
\end{equation}
The only elements $\tilde z \in z + \frac{1}{q} Y$ that we need to
consider are the eight elements
\begin{equation}
  \tilde z = z - \delta_1 \alpha_1 -\delta_3 \alpha_3 - \delta_4 \alpha_4, 
  \quad \delta_i \in \{0,1\}.
\end{equation}
First, consider the case that $(\delta_1, \delta_3) \neq (1,1)$. If
also $(\delta_1, \delta_3) \neq (0,0)$, then $|\langle s_2 s_1
\widetilde{z}, \omega_2 \rangle| = 1$.  If $(\delta_1, \delta_3) =
(0,0)$, then $\langle s_2 \widetilde{z}, \omega_2 \rangle =
\frac{q+1}{q} > 1$.

Next, consider the case that $\delta_1=\delta_3=1$ and
$\delta_4=0$. Then, $s_3 \tilde z = -\frac{p}{q} \alpha_1 +
\frac{p+1}{q} \alpha_3 + \frac{1}{q} \alpha_4$, which is a case we
already considered in the preceding paragraph.

Thus, it remains to consider the case
$\delta_1=\delta_3=\delta_4=1$. Then, $s_4 \tilde z =
-\frac{p}{q} \alpha_1 -\frac{p}{q} \alpha_3 + \frac{p}{q} \alpha_4$.
Hence, $s_3 s_4 \tilde z = -\frac{p}{q} \alpha_1 + \frac{q-1}{q}
\alpha_3 + \frac{p}{q} \alpha_4$.  Finally, $s_2 s_1 s_3 s_4 \tilde z
= \frac{p}{q} \alpha_1 + \frac{p+q-1}{q} \alpha_2 + \frac{q-1}{q}
\alpha_3 + \frac{p}{q} \alpha_4$, and hence $\langle s_2 s_1 s_3 s_4
\tilde z , \omega_2 \rangle \geq 1$.  

Hence, ample polytopes for root systems containing $A_4$ are not
diagonally split for odd $q \geq 3$, and together with the even case
above, they are not diagonally split for any $q \geq 2$.

Next, consider the root system $D_4$.  As in \cite[\S VI.4]{bourbaki},
$\alpha_2$ is the simple root corresponding to the node, and
$\alpha_1, \alpha_3$, and $\alpha_4$ are the other simple
roots. Define the element
\begin{equation}
z = \frac{p}{q} (\alpha_1 + \alpha_3 + \alpha_4).
\end{equation}
Similarly to the $A_4$ case, we only need to consider the elements
\begin{equation}
  \tilde z = z - \delta_1 \alpha_1 - \delta_3 \alpha_3 - \delta_4 \alpha_4, 
  \quad \delta_i \in \{0,1\}.
\end{equation}
If $\delta_1=\delta_3=\delta_4$ then we see that $|\langle s_2 \tilde
z, \omega_2 \rangle| \geq 1$.  For the other cases, using symmetry, we
may assume that $\delta_1 = \delta_3 = 0$ and $\delta_4 = 1$.  Then,
$|\langle s_2 s_4 \tilde z, \omega_2 \rangle | > 1$. Hence, ample
polytopes containing $D_4$ are not diagonally split.

Consider now the root system $B_3$, with simple roots $\alpha_1,
\alpha_2, \alpha_3$, so that $\alpha_2$ corresponds to the central
vertex and $\alpha_3$ is the short simple root. Let $z = \frac{p}{q}
(\alpha_1 + \alpha_3)$. Then, $|\langle s_3 s_2 z, \omega_3 \rangle|
\geq 1$, and the same is true if we replace $z$ by $z - (\alpha_1 +
\alpha_3)$, $s_1 (z - \alpha_1)$, or $s_3 (z - \alpha_3)$.  This
proves the desired inequality, so that ample polytopes containing
$B_3$ are not diagonally split.

Similarly, consider the root system $C_3$, with simple roots
$\alpha_1, \alpha_2, \alpha_3$ such that $\alpha_2$ corresponds to the
central vertex and $\alpha_3$ is the long simple root. Let $z :=
\frac{p}{q} (\alpha_1 + \alpha_3)$. Then, $|\langle s_2 z, \omega_2
\rangle | \geq 1$, and the same is true if we replace $z$ by $z -
(\alpha_1 + \alpha_3)$, $s_1 (z - \alpha_1)$, or $s_3 (z - \alpha_3)$.

Finally, consider the root system $G_2$, and now allow $q \geq 2$ to
be any integer. Let $p = \lfloor \frac{q}{2} \rfloor$.  Let $\alpha_1$
be the short simple root and $\alpha_2$ be the long simple root.
Consider $z := \frac{p}{q} \alpha_2$.  Then, $|\langle s_1 z, \omega_1
\rangle | \geq 1$. The same is true if we replace $z$ by $s_2(z -
\alpha_2)$.  This proves that ample polytopes are not diagonally split
for $G_2$.

This completes the proof that ample polytopes for root systems other
than $A_1, A_2, A_3$, and $B_2$ are not diagonally split for any $q
\geq 2$.  We claim also that ample polytopes are not diagonally split
in the case where $q$ is even and the root system is of type $B_2$.
For this, let $\alpha_1$ be the long simple root and $\alpha_2$ be the
short simple root.  Consider the element $z = \frac{1}{2} \alpha_1$.
Then, the same argument as in the case $G_2$ applies.

It remains to prove the claims that ample polytopes are diagonally
split for odd $q$ in the cases $A_1, A_2, A_3$, and $B_2$, and in the
case of $A_1$ and $A_2$, also for even $q$. For the case $A_1$, this
is clear, and in the case $A_2$, it follows by choosing, for any $z
\in \frac{1}{q} Y / Y$, the representative $\tilde z \in \frac{1}{q}
Y$ such that $\langle \tilde z, \omega_i \rangle \in [0,1)$ for both
fundamental coweights $\omega_i$.  Next, consider the case $B_2$, and
let $q \geq 3$ be odd.  Let $\alpha_1$ be the long root and $\alpha_2$
be the short root.  Then, for any $z \in \frac{1}{q} Y / Y$, choose
the representative $\tilde z \in \frac{1}{q} Y$ such that $|\langle
\tilde z, \omega_1 \rangle| < \frac{1}{2}$, $|\langle \tilde z,
\omega_2 \rangle| < 1$, and $\langle \tilde z, \omega_1 \rangle$ and
$\langle \tilde z, \omega_2 \rangle$ are either both nonnegative or
both nonpositive.  It is easy to verify that $\tilde z \in
\mathbb{F}_P^{\circ} \cap \frac{1}{q} \Gamma$, as required.

Finally, consider the case $A_3$ with $q$ odd. Let
$\alpha_1, \alpha_2$, and $\alpha_3$ be the simple roots, with
$\alpha_2$ corresponding to the central vertex.  Then, for any $z \in
\frac{1}{q} Y / Y$, first suppose that $\langle \tilde z, \omega_2
\rangle$ is integral for all representatives $\tilde z$ of $z$. In
this case, choose the representative $\tilde z$ so that $\langle
\tilde z, \omega_2 \rangle = 0$ and $|\langle \tilde z, \omega_i
\rangle| < \frac{1}{2}$ for $i \in \{1,3\}$.  Otherwise, if $\langle
\tilde z, \omega_2 \rangle$ is not integral for any representative
$\tilde z$ of $z$, choose $\tilde z$ such that $|\langle \tilde z,
\omega_i \rangle | < 1$ for all $i$, either all $\langle \tilde z,
\omega_i \rangle$ are nonnegative or all are nonpositive, and such
that $|\langle \tilde z, \omega_1 + \omega_3 \rangle | \leq 1$ (where
$\alpha_2$ corresponds to the central vertex).  A straightforward
computation verifies that this yields a diagonal splitting.
\end{proof}

\appendix
\section{Sharpening Theorem \ref{main2s} to preserve winning
conditions}\label{s:m2win}
Here we explain that, if one restricts to tuples $(x_1, \ldots, x_m)
\in \Lambda(P_1) \times \cdots \Lambda(P_m)$ such that the $x_i$ are
winning, then restricting the equivalence relation $\sim$ and the root
moves to these tuples, Theorem \ref{main2s} continues to hold:
\begin{thm}\label{main2s-win}
  Suppose that $x \in \Lambda(P_1 + \cdots + P_m)$
  is winning. Then
\begin{enumerate}
\item[(i)] There exists a tuple $(x_1, \ldots, x_m) \in \Lambda(P_1)
  \times \cdots \times \Lambda(P_m)$ of winning elements such that $x
  = x_1 + \cdots + x_m$;
\item[(ii)] If $(x_1, \ldots, x_m), (x_1', \ldots, x_m') \in
  \Lambda(P_1) \times \cdots \times \Lambda(P_m)$ are two tuples of
  winning elements such that $x_1 + \cdots + x_m = x = x_1' + \cdots +
  x_m'$, then the tuples are related by a sequence of root moves that
  only pass through tuples of winning elements.
\end{enumerate}
\end{thm}
This sharpens the theorem, and further explains its proof.
\begin{rem}
  Note that, in contrast to Theorem \ref{main2s} itself,  even when
  $P_1 = \cdots = P_m$, it is not necessarily true that all root
  moves through tuples of winning elements are generated by root moves
  involving only adjacent indices $j = i+1$ in \eqref{eq:equivmove}.
  This is because adjacent elements in a tuple $(x_1, \ldots, x_m)$ of
  winning elements with $x_1 + \cdots + x_m$ winning need not themselves sum
  to a winning element. So, one cannot deduce that there is a sequence
  of root moves between adjacent elements which swaps the two elements while
  only passing through pairs of winning elements.
\end{rem}

The theorem rests on the following observation:
\begin{lemma} \label{stillwinlem} Suppose that $x \in \Lambda(P)$ is
  winning. If $\langle x, \alpha_i^\vee \rangle \geq 1$, then $x -
  \alpha_i \in \Lambda(P)$ is also winning.
\end{lemma}
\begin{proof}
  Suppose that $x \in \Lambda(P)$ is winning and $\langle x,
  \alpha_i^\vee \rangle \geq 1$.  First note that $x - \alpha_i \in
  \Lambda(P)$ by Proposition \ref{mainclaim}.  Next, if $\langle x,
  \alpha_i^\vee \rangle = 1$, then $x - \alpha_i \mapsto x$ under the
  numbers game, so $x-\alpha_i$ is also winning.  Suppose now that
  $\langle x, \alpha_i^\vee \rangle \geq 2$.  By firing vertices other
  than $i$ which are not adjacent to $i$, we may assume that $\langle
  x, \alpha_j^\vee \rangle \geq 0$ whenever $j$ is not adjacent to
  $i$. Since, for $j$ adjacent to $i$, $\langle x, \alpha_j^\vee
  \rangle \geq -1$, and also $\langle \alpha_i, \alpha_j^\vee \rangle
  \leq -1$, it follows that $\langle x - \alpha_i, \alpha_j^\vee
  \rangle \geq 0$ for all $j$ adjacent to $i$.  Moreover, since
  $\langle \alpha_i, \alpha_i^\vee \rangle = 2$, it also follows that
  $\langle x - \alpha_i, \alpha_i^\vee \rangle \geq 0$.  Hence, $x -
  \alpha_i$ is dominant, and therefore winning. (Without assuming that
  $\langle x, \alpha_j^\vee \rangle \geq 0$ whenever $j$ is not
  adjacent to $i$, we then see that $x - \alpha_i$ will be winning,
  but not dominant in general.)
\end{proof}

\subsection{Simple moves}
\begin{defn}
  Let $(x_1, \ldots, x_m) \in X^m$. Suppose that $\alpha$ is a simple
  root and $j, k$ are indices such that $\langle x_j, \alpha^\vee
  \rangle \leq -1$ and $\langle x_k, \alpha^\vee \rangle \geq
  1$. Then, setting $x_j' = x_j +\alpha$ and $x_k = x_k - \alpha$, and
  $x_\ell' = x_\ell$ for $\ell \notin \{j,k\}$, we say $(x_1', \ldots,
  x_m')$ is \emph{obtained from $(x_1, \ldots x_m)$ by a simple move}.
\end{defn}
Note that a simple move is a very special type of root move, in the
case that all the elements of the tuples are in the relevant
polytopes. In fact, it is enough to check this for one of the tuples:
\begin{lemma} \label{l:simple-win1} If $P_1, \ldots, P_m$ are special
  ample polytopes, $(x_1, \ldots, x_m) \in \Lambda(P_1) \times \cdots
  \times \Lambda(P_m)$, and $(x_1', \ldots, x_m')$ is obtained from
  $(x_1, \ldots, x_m)$ by simple moves, then also $(x_1', \ldots,
  x_m') \in \Lambda(P_1) \times \cdots \times \Lambda(P_m)$.
\end{lemma}
The lemma is an immediate consequence of Proposition \ref{mainclaim}.
Furthermore, using Lemma \ref{stillwinlem}, we can prove
\begin{lemma} \label{l:simple-win2} If $(x_1, \ldots, x_m) \in
  \Lambda(P_1) \times \cdots \times \Lambda(P_m)$ is a tuple of
  winning elements, then any simple move results in another tuple of
  winning elements.
\end{lemma}
\begin{proof}
  It is clearly enough to assume $m = 2$.  Without loss of generality,
  the move is $(x_1, x_2) \mapsto (x_1+\alpha, x_2 - \alpha)$ where
  $\langle x_1, \alpha^\vee \rangle < 0$ and $\langle x_2, \alpha^\vee
  \rangle > 0$.  Since $x_1$ is winning, $\langle x_1, \alpha^\vee
  \rangle = -1$ and $x_1+\alpha$ is obtained by playing a move of the
  numbers game. Hence $x_1+\alpha$ is winning.  Also, $x_2-\alpha$ is
  winning by Lemma \ref{stillwinlem}.
\end{proof}
Also, the proof of Lemma \ref{m2lem2} actually implies
\begin{lemma}
  Let $P_1$ and $P_2$ be special ample polytopes, $(x_1, x_2) \in
  \Lambda(P_1) \times \Lambda(P_2)$, $P = P_1 + P_2$, and $x = x_1 +
  x_2 \in \Lambda(P)$.  If $\alpha$ is $P$-progressive for $x$, then
  there is a sequence of simple moves taking $(x_1, x_2)$ to a pair
  $(x_1', x_2')$ such that, for some $i\in\{1,2\}$, $\alpha$ is
  $P_i$-progressive for $x_i'$.
\end{lemma}
Hence, if $(x_1, x_2)$ additionally has the property that $x_1$ and
$x_2$ are winning, Lemma \ref{l:simple-win2} implies that all pairs
obtained along the way from $(x_1, x_2)$ to $(x_1', x_2')$ (along with
$(x_1', x_2')$ itself) consist of winning elements.

\subsection{The case $m=2$ of Theorem
  \ref{main2s-win}.(ii)}\label{ss:meq2-win}
In order to explain the general result, it is best to begin with the
case $m=2$.

\begin{proof}[Proof of Theorem \ref{main2s-win}.(ii) for $m=2$]
  Without the winning conditions, this is the case $m=2$ of Theorem
  \ref{main2s}.(ii).  Let $x = x_1 + x_2$ and let $\alpha$ be
  $P$-progressive for $x$.  We assume the statement for tuples of
  winning elements which sum to $x+\alpha$.  In the proof of Theorem
  \ref{main2s}.(ii), the moves taken are either simple, which preserve
  the property of elements being winning by Lemma \ref{l:simple-win2},
  or else are moves obtained by Lemma \ref{m2lem3} from root moves for
  a tuple whose sum is $x+\alpha$. As in the proof of Theorem
  \ref{main2s}.(ii), the first such move of the latter type begins
  with a tuple $(x_1, x_2)$ of winning elements such that $\alpha$ is
  $P_1$-progressive for $x_1$. By Proposition \ref{p:powin}, then $x_1
  + \alpha$ is winning and in $\Lambda(P_1)$.  By induction on the
  number of such moves required, we can then assume that the latter
  root move is of the form $(x_1 + \alpha, x_2) \mapsto
  (x_1+\alpha+\beta, x_2-\beta)$ for $\beta \in \Delta$ such that
  $x_1+\alpha+\beta$ and $x_2-\beta$ are winning.  Thus, it remains to
  prove the following sharpening of Lemma \ref{m2lem3}.
\end{proof}
\begin{lemma}\label{m2lem3-win}
  Suppose that $(x_1, x_2) \in \Lambda(P_1) \times \Lambda(P_2)$,
  $x_1$ and $x_2$ are winning, and $\alpha$ is a simple root 
  such that $x_1 + \alpha$ is winning and in $\Lambda(P_1)$, and $x =
  x_1 + x_2$ satisfies $\langle x, \alpha^\vee \rangle \geq -1$.  If
  $\beta \in \Delta$ is such that $(x_1+\alpha + \beta, x_2-\beta) \in
  \Lambda(P_1) \times \Lambda(P_2)$ is a tuple of winning elements,
  then either $(x_1 + \beta, x_2 - \beta) \in \Lambda(P_1) \times
  \Lambda(P_2)$ or $(x_1 + \alpha + \beta, x_2 - \alpha - \beta) \in
  \Lambda(P_1) \times \Lambda(P_2)$, and it is a pair of winning
  elements. Furthermore, in the latter case, either $(x_1 + \alpha,
  x_2-\alpha)$ is in $\Lambda(P_1) \times \Lambda(P_2)$ and is a pair
  of winning elements, or else $\alpha+\beta \in \Delta$.
\end{lemma}
\begin{proof}
  To prove the first assertion, we only need to check that, following
  the proof of Lemma \ref{m2lem3}, the final tuple $(x_1 + \beta, x_2
  - \beta)$ or $(x_1 + \alpha + \beta, x_2 - \alpha - \beta)$ is a
  tuple of winning elements.  In the first case, we had $\langle x_1 +
  \alpha + \beta, \alpha^\vee \rangle \geq 1$, and $x_1 + \alpha +
  \beta$ is winning, so Lemma \ref{stillwinlem} implies that $x_1 +
  \beta$ is winning; $x_2 - \beta$ is winning by assumption.  In the
  second case, we had $\langle x_2 - \beta, \alpha^\vee \rangle \geq
  1$ and $x_2 - \beta$ is winning, and hence $x_2-\beta-\alpha$ is
  winning by Lemma \ref{stillwinlem}; $x_1 + \alpha + \beta$ is
  winning by assumption.  For the final statement, we note that, if
  $(x_1 + \alpha, x_2 - \alpha)$ is not winning, then in view of Lemma
  \ref{stillwinlem}, $\langle x_2, \alpha^\vee \rangle \leq 0$, and
  the statement then follows as in the proof of Lemma \ref{m2lem3}.
\end{proof}
\subsection{Proof for general $m$}
As in the proof of Theorem \ref{main2s}, let $P = P_1 + \cdots + P_m$
and $\alpha$ be a simple root that is $P$-progressive for $x$.

(i) This follows from the proof of Theorem \ref{main2s}.(i), if we
notice that, if $(y_1, \ldots, y_m)$ is a tuple of winning elements
with $y_1 + \cdots y_m = x + \alpha$ and $\langle y_i, \alpha^\vee
\rangle \geq 1$, then $y_i -\alpha$ is still winning by Lemma
\ref{stillwinlem}, and hence $(y_1, \ldots, y_{i-1}, y_i - \alpha,
y_{i+1}, \ldots, y_m)$ is a tuple of winning elements summing to $x$.

(ii) We adapt the proof of Theorem \ref{main2s}.(ii).  Let $Q := P_1 +
\cdots + P_{m-1}$, $y = x_1 + \cdots + x_{m-1}$, and $y' = x_1' +
\cdots + x_{m-1}'$. We assume the statement of the theorem for tuples
whose sum is $x$.  The proof below will be slightly more complicated
than the proof of Theorem \ref{main2s}.(ii), because we cannot in
general assume that $y$ is winning, and hence cannot apply the
statement of the theorem to $y$ itself (and in particular, we do not
need to assume the statement of the theorem for smaller values of
$m$).

By performing simple moves, we can assume that $\alpha$ is either
$Q$-progressive for $y$ or $P_m$-progressive for $x_m$.  In the case
it is $Q$-progressive for $y$, we can iterate this procedure on the
tuple $(x_1, \ldots, x_{m-1})$ until there exists an index $i$ such
that $\alpha$ is $P_i$-progressive for $x_i$. So we assume this.  By
doing the same for $(x_1', \ldots, x_m')$, we can suppose that
$\alpha$ is $P_j$-progressive for $x_j'$. By hypothesis, $(x_1,
\ldots, x_{i-1}, x_i + \alpha, x_{i+1}, \ldots, x_m)$ and $(x_1',
\ldots, x_{j-1}', x_j' + \alpha, x_{j+1}', \ldots, x_m')$ are related
by root moves which pass only through winning tuples (since these are
tuples whose sum is $x+\alpha$).

It is then enough to show that, for a single root move $(x_1, \ldots,
x_{i-1}, x_i + \alpha, x_{i+1}, \ldots, x_m) \mapsto (y_1, \ldots,
y_{m})$ (with $x_\ell, y_\ell \in \Lambda(P_\ell)$ winning for all
$\ell$, and $x_i+\alpha \in \Lambda(P_i)$ winning), then there exists
an index $k$ such that $y_k - \alpha$ is in $\Lambda(P_k)$ and
winning, and such that $(x_1, \ldots, x_m)$ is related to $(y_1,
\ldots, y_{k-1}, y_k - \alpha, y_{k+1}, \ldots, y_m)$ by root moves
that pass only through tuples of winning elements. If there exists an
index $k$ such that $x_k=y_k$ and $\langle x_k, \alpha^\vee \rangle
\geq 1$, then the statement follows immediately.  If not, then the
root move is of the form $(x_i + \alpha, x_j) \mapsto (x_i + \alpha +
\beta, x_j - \beta)$ for some $\beta \in \Delta$, and $\langle x_\ell,
\alpha^\vee \rangle \leq 0$ for all $\ell \notin \{i,j\}$.  Since
$\langle x+\alpha, \alpha^\vee \rangle \geq 1$ (as $\alpha$ is
$P$-progressive for $x$), it follows that $\langle (x_i+\alpha)+x_j,
\alpha^\vee \rangle \geq 1$, and hence $\langle x_i + x_j, \alpha^\vee
\rangle \geq -1$.  Now, the statement follows from Lemma
\ref{m2lem3-win} (applied to the pair $(x_i, x_j)$ together with
$\alpha$ and $\beta$); $k$ will then be either $i$ or $j$.

\bibliographystyle{amsalpha}
\bibliography{references}
\end{document}